\documentclass[11pt]{amsart}
\usepackage{hyperref} 
\usepackage{amsmath, amsthm, amssymb, amscd}
\usepackage{enumerate}
\usepackage{verbatim}
\usepackage[noadjust]{cite}
\usepackage{graphicx}
\usepackage{color}

\newcommand{\RR}{\mathbb{R}}
\newcommand{\HH}{\mathcal{H}}

\newcommand{\Lip}{\textnormal{Lip}}

\newcommand{\dist}{\textnormal{dist}}

\newtheorem{thm}{Theorem}[section]
\newtheorem{lemma}[thm]{Lemma}
\newtheorem{cor}[thm]{Corollary}

\theoremstyle{remark}

\theoremstyle{definition}
\newtheorem{definition}[thm]{Definition}

\begin{document}
\title{A note on topological dimension, Hausdorff measure, and rectifiability}
\author{Guy C. David} 
\address[David]{Department of Mathematical Sciences, Ball State University, Muncie, IN 47306}
\email{gcdavid@bsu.edu}

\author[Enrico Le Donne]{Enrico Le Donne}

\address[Le Donne]{Department of Mathematics and Statistics, 
University of Jyv\"askyl\"a, 40014 Jyv\"askyl\"a, Finland}
\email{enrico.ledonne@jyu.fi}
\email{ledonne@msri.org}
\date{July 7, 2018}

\thanks{G.C.D was partly funded by NSF DMS-1758709.
E.L.D. was partially supported by the Academy of Finland 
(grant 288501 `\emph{Geometry of subRiemannian groups}')
and by the European Research Council 
(ERC Starting Grant 713998 GeoMeG `\emph{Geometry of Metric Groups}').
}

\maketitle
	
\section{Introduction}
	
The purpose of this note is to record a consequence, for general metric spaces, of a recent result of Bate \cite{Bate}. We prove the following fact:

\begin{thm}\label{mainthm}
Let $X$ be a compact metric space of topological dimension $n$. Suppose that the $n$-dimensional Hausdorff measure of $X$, $\HH^n(X)$, is finite. 

Suppose further that
\begin{equation}\label{cj1}
\liminf_{r\rightarrow 0} \frac{\HH^n(B(x,r))}{r^n} > 0 \text{ for } \HH^n\text{-a.e. } x\in X.
\end{equation}

Then $X$ contains an $n$-rectifiable subset of positive $\HH^n$-measure.

Moreover, assumption \eqref{cj1} is unnecessary if one uses recently announced results of Cs\"ornyei-Jones.
\end{thm}

The use in Theorem \ref{mainthm} of the results of Cs\"ornyei-Jones arises purely through our use of the work of Bate \cite{Bate} (Theorem \ref{bate} below), and does not directly appear in any of the proofs here. 
See Bate's discussion just below \cite[Theorem 1.1]{Bate} for details concerning the announcement of Cs\"ornyei-Jones and the dependence of Theorem \ref{bate} on them. 

When $X$ is a subset of some Euclidean space, Theorem \ref{mainthm}, without assuming \eqref{cj1} or the results of Cs\"ornyei-Jones, appears to already be known (see \cite[p. 880]{SemmesICM}), as a consequence of the Besicovitch-Federer projection theorem. For general metric spaces, the Besicovitch-Federer theorem is unavailable \cite{BCW}, but Bate's work \cite{Bate} serves as our replacement. 

When $n=1$, Theorem~\ref{mainthm} (without assuming \eqref{cj1} or relying on the results of Cs\"ornyei-Jones) is a consequence of the fact that continua of finite $\HH^1$-measure are Lipschitz images of $[0,1]$ (see, e.g., \cite[Lemma 3.7]{Schul}), but this particular fact does not extend to $n>1$.

We now recall some background: For compact metric spaces, the commonly used notions of topological dimension (Lebesgue covering dimension, large/strong inductive dimension, and small/weak inductive dimension) agree. We refer the reader to \cite[Sections I.4 and II.5]{Nagata} for this fact and the relevant definitions. For Hausdorff measure and dimension, we refer the reader to \cite[Chapter 8]{He}.

An $\HH^n$-measurable subset $E$ of a metric space $X$ is called \textit{$n$-rectifiable} if
$$ \HH^n(E\setminus \bigcup_{i=1}^\infty f_i(F_i)) =0 $$
where $F_i$ are measurable subsets of $\RR^n$ and $f_i:F_i\rightarrow X$ are Lipschitz maps. By a theorem of Kirchheim \cite[Lemma 4]{Kirchheim}, one can equivalently take $f_i$ to be bi-Lipschitz mappings.  

A subset $E$ of a metric space $X$ is called \textit{purely $n$-unrectifiable} if it contains no $n$-rectifiable subsets of positive $\HH^n$-measure.

If a compact metric space $X$ has topological dimension $n$, then it is a well-known fact (see, e.g., \cite[Theorem 8.15]{He}) that $\HH^n(X)>0$, although certainly $X$ may have infinite $n$-dimensional Hausdorff measure or even Hausdorff dimension strictly larger than $n$, as is the case for classical fractals. Thus, Theorem~\ref{mainthm} says that in the extremal situation, one must see some Euclidean structure in the space.

Related results, in which a combination of $n$-dimensional topological behavior and $n$-dimensional measure theoretic behavior implies some type of rectifiability, can be found, for example, in \cite{DS93, DS00, SemmesICM, JKV, GCD16}. These results typically employ more quantitative assumptions to obtain more quantitative conclusions than our Theorem~\ref{mainthm}.

It is easy to see that the assumptions of Theorem~\ref{mainthm} (including \eqref{cj1}) do not imply $n$-rectifiability of the whole space $X$. For example, $X$ may be the disjoint union of the unit ball in $\mathbb{R}^n$ with a metric space that is a purely $n$-unrectifiable Cantor set of positive $n$-dimensional Hausdorff measure. 

More surprising is that the assumptions of Theorem~\ref{mainthm} do not imply $n$-rectifiability even if one assumes that $X$ is a compact $n$-dimensional topological manifold. In the appendix to \cite{SW}, Schul and Wenger construct a compact topological $n$-sphere with $\mathcal{H}^n(X)<\infty$ that contains a purely $n$-unrectifiable subset of positive measure.

On the other hand, Theorem~\ref{mainthm} implies that every open ball in a compact $n$-manifold with finite $\HH^n$-measure contains an $n$-rectifiable subset of positive $\HH^n$-measure. Note that there exist such manifolds with no bi-Lipschitz embedding into any Euclidean space \cite{Laakso, Semmes2}.
	
\section{Proof of Theorem \ref{mainthm}}
Given a metric space $X$ and $m\in\mathbb{N}$, let $\Lip_1(X,m)$ denote the space of bounded, $1$-Lipschitz functions $f\colon X\rightarrow \RR^m$, equipped with the supremum distance, which we denote $\dist$. This is a complete metric space, and hence residual subsets (in the sense of Baire category) are dense.

The proof of Theorem~\ref{mainthm} is based on the following recent result.

\begin{thm}[{Bate \cite[Theorem 1.1]{Bate}}]\label{bate}
Let $X$ be a complete, purely $n$-unrectifiable metric space with $\HH^n(X)<\infty$.  Suppose further that \eqref{cj1} holds.

Then the set of all $f\in Lip_1(X,m)$ with $\HH^n(f(X))=0$ is residual.

Moreover, assumption \eqref{cj1} is unnecessary if one uses recently announced results of Cs\"ornyei-Jones.
\end{thm}

This has the following easy consequence.
\begin{cor}\label{cor}
Let $X$ be a compact, purely $n$-unrectifiable metric space with $\HH^n(X)<\infty$ and satisfying \eqref{cj1}. 

Let $g:X\rightarrow [0,1]^n$ be continuous. Then there is a sequence of Lipschitz functions $f_i:X\rightarrow [0,1]^n$ that converge to $g$ in the supremum distance and satisfy $\HH^n(f_i(X))=0$ for all $i\in\mathbb{N}$.

Moreover, assumption \eqref{cj1} is unnecessary if one uses recently announced results of Cs\"ornyei-Jones.
\end{cor}
\begin{proof}
Let $h_i$ be a sequence of $L_i$-Lipschitz functions converging to $g$ in the supremum distance. (The existence of such a sequence is a consequence of the Stone-Weierstrass theorem, or see \cite[Lemma 2.4]{Semmes} for a simple direct proof.) Thus $L_i^{-1} h_i \in \Lip_1(X,n)$.

By Theorem~\ref{bate}, we can find, for each $i\in\mathbb{N}$, a Lipschitz function $g_i\in \Lip_1(X,n)$ satisfying
$$ \dist(L_i^{-1} h_i , g_i) < L_i^{-1} 2^{-i} \text{ and } \HH^n(g_i(X)) = 0.$$

Consider the $1$-Lipschitz retraction $r:\RR^n \rightarrow  [0,1]^n$ given by
\begin{equation}\label{retraction}
r(x_1,x_2,\dots,x_n) = (\psi(x_1), \psi(x_2), \dots, \psi(x_n)),
\end{equation}
where 
\[\psi(t) =  \begin{cases} 
      0 & t<0 \\
			t & 0\leq t \leq 1 \\
      1 & t>1.
   \end{cases}
\]


Lastly, set
$$ f_i = r \circ (L_i g_i).$$
Since $r$ is Lipschitz and $\HH^n(g_i(X))=0$, we have  $\HH^n(f_i(X))=0$ for all $i\in\mathbb{N}$. Furthermore,
$$ \dist(f_i, g) = \dist(r\circ (L_i g_i), r\circ g) \leq \dist(L_i g_i , g) <  2^{-i}+ \dist(h_i , g) \rightarrow 0.$$

\end{proof}

To prove Theorem~\ref{mainthm}, we will also need some topological information.

\begin{definition}
Let $f\colon X\rightarrow Y$ be a continuous map between metric spaces. A point $y\in Y$ is called a \textit{stable value} of $f$ if there is $\epsilon>0$ such that $y\in g(X)$ for every continuous $g\colon X\rightarrow Y$ with $\dist(g,f)<\epsilon$.
\end{definition}

Some basic and well-known facts about stable values of mappings to $[0,1]^n$ are collected in the following lemma.
\begin{lemma}\label{stableproperties}
Let $X$ be a metric space and let $y$ be a stable value of a continuous map $f\colon X\rightarrow [0,1]^n$. Then
\begin{enumerate}[(i)]
\item $y\notin\partial\left([0,1]^n\right)$,
\item $y$ is a stable value of $g$ for each continuous $g:X\rightarrow [0,1]^n$ with $\dist(g,f)$ sufficiently small, and 
\item $f(X)$ contains an open neighborhood of $y$ in $[0,1]^n$. 
\end{enumerate}
\end{lemma}
\begin{proof}
Part (i) is simple and explained in \cite[Example VI 4]{HW}. Part (ii) is an immediate consequence of the definition of stable value.

For part (iii), recall the $1$-Lipschitz retraction $r\colon\RR^n \rightarrow  [0,1]^n$ defined in \eqref{retraction}. Note that $r$ maps $\RR^n\setminus [0,1]^n$ onto the boundary of $[0,1]^n$.

Let $y$ be a stable value of $f:X\rightarrow [0,1]^n$, with parameter $\epsilon>0$. Then, by part (i), $y\in (0,1)^n$. 
We claim that $f(X)$ contains $B(y,\epsilon) \cap [0,1]^n$. Consider any $y'\in B(y,\epsilon)\cap [0,1]^n$. The formula
$$ h(x) = r(x+y-y')$$
defines a continuous map from $[0,1]^n$ to itself such that $h(y') = y$ and $|h(x)-x|<\epsilon$ for all $x\in [0,1]^n$.

Consider the map $g:X\rightarrow [0,1]^n$ defined by $g=h\circ f$. Then $\dist(f,g)<\epsilon$, so $g(x)=y$ for some $x\in X$. Therefore,
$$ r(f(x) + y - y') = y.$$
Since $y$ is not on the boundary of $[0,1]^n$, we must have $f(x) + y-y'= y$, i.e., $f(x)=y'$. 
\end{proof}

The following theorem is the second main ingredient in the proof of Theorem~\ref{mainthm}.
\begin{thm}[Theorem III.1 of \cite{Nagata}]\label{stable}
Let $X$ be a compact metric space of topological dimension $n$. Then there is a continuous map $g\colon X\rightarrow [0,1]^n$ with a stable value.
\end{thm}

The technique of using stable values to find some rectifiable structure in a metric space was used by David and Semmes \cite[Section 12.3]{DS} and Bonk and Kleiner \cite{BK} in similar contexts.

\begin{proof}[Proof of Theorem~\ref{mainthm}]

Let $X$ be a compact metric space of topological dimension $n$ and $\HH^n(X)<\infty$. We claim that $X$ contains an $n$-rectifiable subset of positive measure. Suppose, to the contrary, that $X$ is purely $n$-unrectifiable.

Then, by Theorem~\ref{stable}, there is a continuous map $g\colon X\rightarrow [0,1]^n$ with a stable value $y$. By Corollary~\ref{cor}, there is a sequence $f_i$ of Lipschitz maps from $X$ to $[0,1]^n$ that converge to $g$ in the supremum distance and satisfy
\begin{equation}\label{zeromeasure}
\HH^n(f_i(X)) = 0
\end{equation}
for all $i\in\mathbb{N}$.

On the other hand, by Lemma~\ref{stableproperties}(ii), when $i\in\mathbb{N}$ is sufficiently large, the map $f_i$ must also have $y$ as a stable value. In that case, $g_i(X)$ contains an open subset of $[0,1]^n$, by Lemma~\ref{stableproperties}(iii). This contradicts \eqref{zeromeasure}.
\end{proof}

\bibliography{topdimbib}{}
\bibliographystyle{plain}

\end{document}